\documentclass[a4paper,11pt]{amsart}

\usepackage{amsmath}
\usepackage{amsthm}
\usepackage{amsfonts}
\usepackage{amssymb}
\usepackage{graphicx}

\newcommand\eps{\varepsilon}
\newcommand\Prb{\mathbb{P}}
\newcommand\N{\mathbb{N}}

\newcommand\R{\mathbb{R}}

\renewcommand{\le}{\leqslant}
\renewcommand{\ge}{\geqslant}

\title{Subtended Angles}
\author[P.~Balister]{Paul Balister}
\address{Department of Mathematical Sciences, University of Memphis, Memphis TN 38152, USA.}
\email{pbalistr@memphis.edu}
\author[B.~Bollob\'as]{B\'ela Bollob\'as}
\address{Department of Pure Mathematics and Mathematical Statistics,
Wilberforce Road, Cambridge CB3\thinspace0WB, UK; {\em and\/}
Department of Mathematical Sciences, University of Memphis, Memphis TN 38152,
USA; {\em and\/} London Institute for Mathematical Sciences, 35a South St.,
Mayfair, London W1K\thinspace2XF, UK.}
\email{bollobas@dpmms.cam.ac.uk}
\thanks{The first and second author are partially supported by
NSF grant DMS~1301614 and MULTIPLEX no.\ 317532.}
\author[Z.~F\"uredi]{Zolt\'an F\"uredi}
\address{R\'enyi Institute of Mathematics, Budapest, Hungary.}
\email{z-furedi@illinois.edu}
\thanks{The third author's research supported in part by the Hungarian
  National Science Foundation OTKA 104343, by the Simons Foundation
  Collaboration Grant \#317487, and by the European Research Council
  Advanced Investigators Grant 267195.}
\author[I.~Leader]{Imre Leader}
\address{Department of Pure Mathematics and Mathematical Statistics,
Wilberforce Road, Cambridge CB3 0WB, UK.}
\email{I.Leader@dpmms.cam.ac.uk}
\author[M.~Walters]{Mark Walters}
\address{School of Mathematical Sciences, Queen Mary University of London, London E1 4NS, UK.}
\email{\tt  m.walters@qmul.ac.uk}


\begin{document}

\newtheorem{theorem}{Theorem}
\newtheorem{lemma}[theorem]{Lemma}
\newtheorem{proposition}[theorem]{Proposition}
\newtheorem{corollary}[theorem]{Corollary}
\newtheorem*{defn}{Definition}
\newtheorem*{theorem*}{Theorem}
\newtheorem*{lemma*}{Lemma}
\newtheorem*{quotedresult}{Lemma}
\newtheorem*{conjecture}{Conjecture}
\newtheorem{question}{Question}
\theoremstyle{remark}
\newtheorem*{remark}{Remark}

\begin{abstract}
We consider the following question. Suppose that $d\ge2$ and $n$ are
fixed, and that $\theta_1,\theta_2,\dots,\theta_n$ are $n$ specified
angles. How many points do we need to place in $\R^d$ to realise all
of these angles?

A simple degrees of freedom argument shows that $m$ points in $\R^2$
cannot realise more than $2m-4$ general angles.  We give a
construction to show that this bound is sharp when $m\ge 5$.

In $d$ dimensions the degrees of freedom argument gives an upper bound
of $dm-\binom{d+1}{2}-1$ general angles. However, the above result
does not generalise to this case; surprisingly, the bound of $2m-4$
from two dimensions cannot be improved at all. Indeed, our main result
is that there are sets of $2m-3$ of angles that cannot be realised by
$m$ points in any dimension.
\end{abstract}

\maketitle

\section{Introduction}
We consider the following question. Suppose that $d\ge2$ and $n$ are
fixed, and that $\theta_1,\theta_2,\dots,\theta_n$ are $n$ specified
angles. How many points do we need to place in $\R^d$ to realise all
of these angles? (We say an angle $\theta$ is \emph{realised} if there
exist points $A$, $B$ and $C$ such that $A\widehat BC=\theta$. To
avoid trivial cases we assume that all angles lie strictly between $0$
and $\pi$.)

There is a natural `degrees of freedom' argument. Suppose we place $m$
points. Then there are $d$ degrees of freedom for each point, so $dm$
in total, except we have to consider all similarities of $\R^d$. There
are $d$ degrees for translation, one for scaling, and then
$\binom{d}{2}$ for the orthogonal group of isometries. Thus in total
there are $dm-\binom{d+1}{2}-1$ degrees of freedom, and we cannot hope
to realise more than this many angles in general. (We give a formal
proof of this fact in Section~\ref{s:2d-lower}.)

However, it is far from clear that we can realise this many, although
one might guess that they can be realised, at least for large
$n$. Indeed, this is the case in two dimensions.
\begin{theorem}\label{t:2d-opt}
  Suppose that $m\ge 5$ and $n\le 2m-4$. Then, given any $n$ distinct
  angles, there is an arrangement of $m$ points in the plane realising
  all $n$ of these angles. Moreover, these points may
    be chosen in convex position.
\end{theorem}
It is trivial to see that, even without the convex condition, the
theorem does not hold for $m = 3$: the configuration must be a
triangle so two angles can be realised if and only if their sum is
less than $\pi$. Thus, if the angles are chosen independently from the
uniform distribution on $(0,\pi)$, the probability of realising two
angles is $1/2$. F\"uredi and Szigeti~\cite{FurSzi} showed that the
probability of realising four independent uniform $(0,\pi)$ angles by
four points is $79/84$. This is less than~$1$, so
Theorem~\ref{t:2d-opt} is also false for $m = 4$.  One can also check
that it is impossible for four points to represent four distinct
angles $\theta_1>\theta_2>\theta_3>\theta_4$ when $\theta_4>(2/3)\pi$
and $\theta_2+\theta_3>\pi +\theta_4$, even in three dimensions.

Having seen that degrees of freedom gives the correct answer in two
dimensions, it is natural to guess that the same is true in higher
dimensions. Somewhat surprisingly this is not the case: it is not
possible to guarantee to realise \emph{any} more angles than in two
dimensions. Indeed, there are sets of $2m-3$ angles that cannot be
realised by $m$ points in any dimension.

\begin{theorem}\label{t:main-hd}
  Suppose that $m\ge 2$ and that $n= 2m-3$. Then there exists a set of
  $n$ (distinct) angles such that no arrangement of $m$ points in any
  dimension realises all angles in the set.
\end{theorem}
This theorem says that we could not guarantee to achieve more than
$2m-4$ general angles with $m$ points even in arbitrarily large
dimension. Our example of an unachievable set contains angles very
close to either $0$ or~$\pi$. Thus, it is natural to ask if we can do
better if the angles are bounded away from $0$ and $\pi$. Of course,
the degrees of freedom bound will still hold so there will be no
change in two dimensions but, in higher dimensions, there might
be. Our final result shows that constraining the angles away from $0$
and $\pi$ makes a huge difference: in this case we can nearly obtain
the degrees of freedom bound in any dimension.

\begin{theorem}\label{t:away-from-zero}
  Suppose that $d$ and $\eps$ are fixed. Then there exists a
  constant~$c$ such that any $n=dm-c$ angles, all lying between $\eps$
  and $\pi-\eps$, can be realised using $m$ points.
\end{theorem}

There are two trivially equivalent phrasings of our question: we can
ask how many arbitrary angles can be realised by $m$ points, or how
many points do we need to realise $n$ arbitrary angles. It turns out
to be easier to consider the first formulation, and we shall use that
form when discussing upper or lower bounds (i.e., an `upper bound' is an
upper bound on the number of angles that can be realised by $m$
points).

The layout of the paper is as follows. In Section~\ref{s:2d-upper} we
prove the upper bound in two dimensions (Theorem~\ref{t:2d-opt}), and
in Section~\ref{s:2d-lower} the lower bound in two dimensions. Then in
Section~\ref{s:hd-upper} we prove the higher dimensional upper bound
(Theorem~\ref{t:main-hd}). Finally, in Section~\ref{s:away-from-zero}
we prove Theorem~\ref{t:away-from-zero}. We conclude with some open
questions.

\bigskip We end this section with some background.  These
problems were posed by F\"uredi and Szigeti~\cite{FurSzi}; however
there is a long history of related problems, both for realising angles
and for realising distances.

Erd\H{o}s~\cite{MR0015796,MR1531563} asked how many points can be
placed in $\R^d$ such that no angle greater than $\pi/2$ is realised;
the hypercube provides an obvious lower bound of $2^d$ and Erd\H{o}s
asked whether this was extremal. This was answered positively by
Danzer and Gr\"unbaum in~\cite{MR0138040}, who also conjectured that
if all the angles must be acute (rather than just non-obtuse) then
there could not be a superlinear number of points. Erd\H{o}s and
F\"uredi~\cite{MR841305} disproved this in a very strong sense: they
showed that there can be exponentially many points in $\R^d$ with all
realised angles acute. More recently, Harangi~\cite{MR2837593}
significantly improved the exponent in this lower bound. Also, in
\cite{MR527745}, Conway, Croft, Erd\H{o}s and Guy initiated the study
of the more general questions of how many (or few) angles can have
size at least (at most) $\alpha$ for arbitrary angles~$\alpha$.

Of course, our question is somewhat different. These latter questions
are asking for a set all of whose angles have some property
(e.g., are non-obtuse) whereas here we are asking for some of the angles to
take specific values.

There is a vast literature on the many closely related questions
asking about distances rather than angles. Indeed, there are far too
many references for us to do justice to here, and we just mention a
few; see, for example, Brass, Moser and Pach~\cite{MR2163782} for a
more complete survey. These questions date back to a question
of Hopf and Pannwitz~\cite{hopf-pannwitz}, who asked how many times an
$n$ point set in $\R^2$ with diameter~1 can realise the distance~1,
with many people submitting solutions to the journal.  This led
Erd\H{o}s~\cite{MR0015796} to ask at least how many different
distances must be realised by $n$ points in $\R^2$, and he gave a
simple argument giving a lower bound of order $\sqrt
n$. Moser~\cite{MR0046663} gave the first of many improvements to the
exponent in the lower bound. This work has culminated with the very
recent breakthrough of Guth and Katz~\cite{MR3272924} who showed that
at least $\Omega (n/\log n)$ distances must be realised.

\section{The upper bound in two dimensions}\label{s:2d-upper}

We start by showing that if we can realise a subset of the angles that
includes the maximum angle `optimally' then we can realise all the angles
optimally.
\begin{lemma}\label{l:extend-2d}
  Suppose that $m\in \N$, that $\theta_1,\theta_2,\dots,\theta_{2m-4}$
  are $2m-4$ distinct angles, and that $P$ is an arrangement of $m'<m$
  points realising $2m'-4$ of the angles including the maximum
  angle. Then there is an arrangement of~$m$ points realising all
  $2m-4$ angles.
\end{lemma}
\begin{proof}
  Suppose $\theta_1$ is the maximum angle. We show that we can add a
  single point realising any two of the remaining angles.  By doing
  this repeatedly we obtain the result.
  \begin{figure}
    \begin{center}
      \includegraphics[scale=0.75]{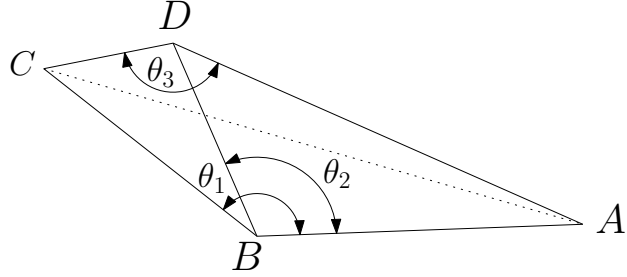}
    \end{center}
    \caption{Adding a point to an existing configuration.}\label{f:2d-ext}
  \end{figure}

  Suppose the two angles we want to add are $\theta_2$ and $\theta_3$.
  Let $A,B$ and $C$ be three points of the configuration such that
  $A\widehat BC$ realises the angle $\theta_1$; see Figure~\ref{f:2d-ext}.  We
  add a point $D$ on the ray (i.e., half-line) $BD$ such that $A\widehat BD=\theta_2$ and
  $A\widehat DC=\theta_3$. Since $\theta_1$ is the maximum angle, the ray $BD$
  lies between the side $BA$ and $BC$. By varying the point $D$ on
  that ray between the point of intersection of the ray and $AC$, and
  the point at infinity we can make $A\widehat DC$ any angle between
  $\pi$ and $0$; in particular we can obtain $\theta_3$.
\end{proof}
We remark that we could realise the angle $\theta_3$ as $E\widehat DF$
for any pair of points $E$ and $F$ lying on opposite sides of the ray
$BD$ in the above construction. We shall use this extra freedom when
proving that the points may be chosen in convex position in
Theorem~\ref{t:2d-opt}.

Using this lemma we can extend an optimally realised subset to the
full set. Thus, the key step in proving Theorem~\ref{t:2d-opt} (except
for the convex condition) is showing that we can realise any six
angles with five points.
\begin{lemma}\label{l:5points-2d}
  Let $\theta_1,\theta_2,\dots,\theta_6$ be six distinct angles. Then
  there is an arrangement of five points in the plane realising all
  six of the angles.
\end{lemma}
\begin{proof}
  We suppose that the angles $\theta_i$ are given in decreasing
  order. Our aim is to show that these angles can be realised by the
  arrangement shown in Figure~\ref{f:paul-2d} in which $A\widehat
  CE=\theta_1$, $B\widehat CE=\theta_2$, $A\widehat BD=\theta_3$,
  $A\widehat BC=\theta_4$, $B\widehat CD=\theta_5$ and $C\widehat
  DE=\theta_6$.
  \begin{figure}
    \begin{center}
      \includegraphics[scale=0.75]{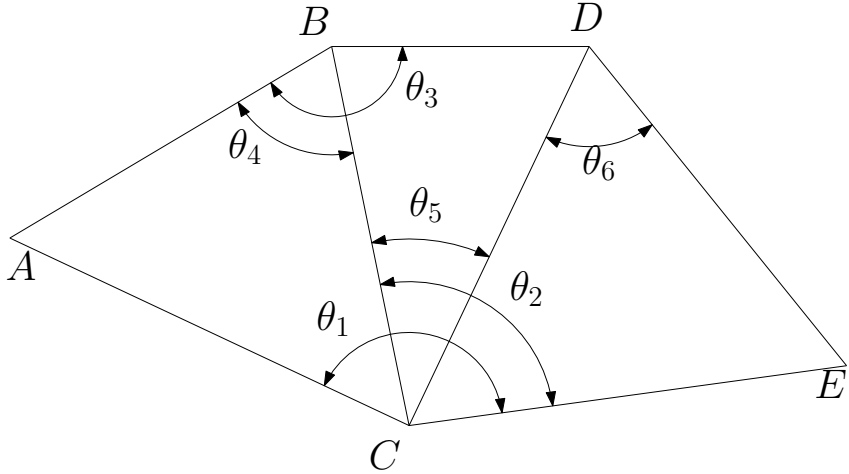}
    \end{center}
    \caption{Five points realising six angles.}\label{f:paul-2d}
  \end{figure}
  To prove that this realisation is possible we start by placing
  $CE$. We then define the ray $CD$ using $\theta_2-\theta_5$, and
  then the point $D$ using $\theta_6$; the ray $CB$ using $\theta_2$
  then the point $B$ using $\theta_3-\theta_4$; the ray $CA$ using
  $\theta_1$ then the point $A$ using $\theta_4$.  By the choice of
  the ordering of the angles the rays $CD$, $CB$ and $CA$ are as in
  Figure~\ref{f:paul-2d}.

  To complete the proof all we have to check is that $CDE$, $BCD$ and
  $ABC$ are all valid triangles.  We do this by checking that the
  angles we have defined sum to less than $\pi$. In $CDE$ we have
  $C\widehat DE+E\widehat CD=\theta_2-\theta_5+\theta_6<\pi$; in $BCD$ we have
  $B\widehat CD+C\widehat BD=\theta_5+\theta_3-\theta_4<\pi$; and in $ABC$ we have
  $A\widehat BC+A\widehat CB=\theta_1-\theta_2+\theta_4<\pi$.
\end{proof}

\begin{proof}[Proof of Theorem~\ref{t:2d-opt}]
  Lemmas~\ref{l:extend-2d} and~\ref{l:5points-2d} prove
  Theorem~\ref{t:2d-opt} except for showing that we may insist that
  the points lie in convex position. To prove this we modify our
  construction slightly.

  Start with the configuration given by Lemma~\ref{l:5points-2d} where
  $\theta_1$ is the largest angle and
  $\theta_2,\theta_3,\dots,\theta_6$ are the five smallest angles, and
  suppose that the remaining angles are $\phi_1,\phi_2,\dots ,\phi_k$
  in increasing order. We add these angles on rays between $CA$ and
  $CB$. Place a point $F$ such that $F\widehat CE=\phi_1$ and
  $A\widehat F B=\phi_2$. We just need to check that this point $F$ is
  convex position. Let $F'$ be the intersection of the (extended)
  lines $CF$ and $BD$. Then $A\widehat{F'}B=A\widehat{F'}D<A\widehat
  BD<\phi_2$. Thus $F$ must lie closer to $C$ than $F'$ and hence is
  in convex position. Repeating this construction replacing $B$ with
  $F$ and $D$ with $B$ for the remaining angles gives the
  result.
\end{proof}

Finally, we consider the case where some angles may be repeated; i.e.,
we have a tuple of angles, rather than a set. We view an angle as
being realised multiple times if it is realised by multiple distinct
pairs of rays. Note that this is more restrictive than just requiring
distinct triples; in particular, we view collinear points giving rise
to the same ray as being the same realisation of an angle. This result
follows easily from our proofs so far, but is a little technical. We
only need one more geometric idea: the following folklore lemma
showing that we can efficiently realise one angle several times.
\begin{lemma}\label{l:eq-angles}
  Suppose that $\theta$ is any angle. Then we can realise $\theta$
  with multiplicity $t(t-1)$ using $2t$ points.
\end{lemma}
\begin{proof}
  Place $t$ pairs of points $x_1,y_1,x_2,y_2,\dots,x_t,y_t$ on a
  circle such that all the $x$'s occur followed by all the $y$'s and
  make the distance between each pair $x_i,y_i$ such that the angle
  subtended by any point on the circle is $\theta$. Then any chord and
  any point between the end points of the chords subtends angle
  $\theta$.  Thus, the angle $\theta$ is realised $t(t-1)$ times.
\end{proof}
Since this is quadratic in the number of points (as opposed to the
linear bound given by Theorem~\ref{t:2d-opt}) we expect the case of
repeated angles to be easier, at least for large numbers of points or
angles. As we shall show, this is indeed the case. However, a little
thought shows that for small numbers of points the reverse may be
true; for example, five points can only realise an angle of $\pi-\eps$
four times.

\begin{lemma}\label{l:combine-tuples}
  Suppose that $\theta=(\theta_1,\theta_2,\dots, \theta_n)$ and
  $\phi=(\phi_1,\phi_2,\dots,\phi_{n'})$ are two tuples of angles realised by
  $m$ and $m'$ points respectively. Then the union of the two tuples
  can be realised with $m+m'-2$ points.
\end{lemma}
\begin{proof}
  This is essentially trivial: we reuse two of the points realising
  $\theta$ when realising $\phi$. We just need to be a little careful
  to avoid coincident rays.

  Fix two points $x$, $y$ of the configuration realising $\theta$ such
  that $xy$ is a line segment of the convex hull of the
  configuration. By a similarity transformation we may assume that
  $x=(0,0)$, $y=(1,0)$ and that all the other points lie in the lower
  half plane.

  Similarly, by a similarity, we may assume that all the points
  realising $\phi$ lie in the upper half plane and that $(0,0)$ and
  $(1,0)$ are two of these points.

  Then, the union of these points sets has size $m+m'-2$ and since we
  have no other coincident points all rays occurring in any of the
  realised angles, except rays contained in the $x$-axis, are
  distinct. Thus, this a realisation of the union of the tuples
  $\theta$ and $\phi$.
\end{proof}
Next we prove a weak version of our result for tuples. Somewhat
surprisingly it is easy to deduce a tight bound from this weak
result. Note we have made no effort to optimise the constant in this
lemma or the subsequent proposition.
\begin{lemma}\label{l:eq-weak}
  Suppose that $n$ and $m$ satisfy $m\ge n/2+30$, and
  $\theta_1,\theta_2,\dots,\theta_n$ are any $n$ (not necessarily
  distinct) angles. Then, we can realise all the angles with $m$
  points.
\end{lemma}
\begin{proof}
  This is a simple induction. Suppose $m\ge n/2+30$. If there are at
  least six distinct angles then, by the inductive hypothesis, we can
  realise the remaining $n-6$ angles with $m-3$ points and, by
  Lemma~\ref{l:5points-2d}, we can realise these six distinct angles
  with five points.  Thus, by Lemma~\ref{l:combine-tuples}, we can
  realise all the original $n$ angles with $m-3+5-2=m$ points
 
  On the other hand suppose any angle $\theta$ occurs with
  multiplicity at least twelve. By the inductive hypothesis, we can
  realise the $n-12$ tuple of the angles with this angle's
  multiplicity reduced by twelve using $m-6$ points, and by
  Lemma~\ref{l:eq-angles} with $t=4$, we can realise $\theta$ twelve
  times with eight points.  Thus, by Lemma~\ref{l:combine-tuples}, we
  can realise all the original $n$ angles with $m-6+8-2=m$ points.
 
  Thus, the only remaining case is when no angle occurs more than
  eleven times and there are no more than five distinct angles; so, in
  particular, $n\le 55$. Trivially we can realise these $n$ angles
  with at most $n+2<n/2+30=m$ points (three points for the first angle
  and one extra point for each extra angle). Thus, in this case the
  results holds and the proof of the lemma is complete.
\end{proof}
\begin{proposition}
  There exists $m_0$ such that, for any $m\ge m_0$, any $n\le 2m-4$,
  and any (not necessarily distinct) angles
  $\theta_1,\theta_2,\dots,\theta_n$, there is an arrangement of $m$
  points in the plane realising all $n$ angles.
\end{proposition}
\begin{proof}
  Suppose that some angle occurs at least 110 times. Then we can
  realise this angle 110 times using 22 points ($t=11$ in
  Lemma~\ref{l:eq-angles}). By Lemma~\ref{l:eq-weak}, we can realise
  the $n-110$ tuple of all the remaining angles using at most
  $m-55+30$ points. In total this uses at most $m-3$ points and the
  proof is complete in this case.

  Thus, the only remaining case is where no angle occurs more than 110
  times. Providing $m_0$, and so $n$, is large enough we can partition
  the remaining angles into sets of angles $A_1,A_2,\dots,A_k$ where
  each $A_i$ has size at least six, each $A_i$ consists of distinct
  angles, and at most one $A_i$ has odd size. Using
  Theorem~\ref{t:2d-opt} we can realise each set $A_i$ of angles with
  $\big\lceil |A_i|/2\big\rceil+2$ points. Thus, by
  Lemma~\ref{l:combine-tuples} repeatedly, we can realise all $n$
  angles with
  \[2+\sum_{i=1}^k\big \lceil
  |A_i|/2\big \rceil=2+\left\lceil\sum_{i=1}^k|A_i|/2\right\rceil=2+\lceil
  n/2\rceil\le m
  \]
  points (where the first equality used the fact that at most one
  $|A_i|$ is odd).
\end{proof}
\section{The lower bound in two dimensions}\label{s:2d-lower}

In the next section we prove a result (Corollary~\ref{c:2m-3}) which
formalises the degrees of freedom intuition we gave in the
introduction and shows that we can almost never realise a set of
$2m-3$ of angles with $m$ points in the plane. In fact, we prove a
stronger result as we shall need this in the next section where we
deal with the higher dimensional case.

We start with a simple algebraic observation (see, for example,
Chapter 18 of~\cite{Garling-Galois}).
\begin{theorem}\label{t:trans-deg}
  Suppose that $\phi_i(x_1,..,x_{d-1})$, $i=1,...,d$, are rational
  functions. Then there exists a (non-zero) polynomial $g$ in
  $\R[z_1,\dots,z_d]$ such that $g(\phi_1,...,\phi_d)=0$.
\end{theorem}
\begin{proof}
  The rational functions $\phi_i$ represent $d$ elements in the field
  $\R(x_1,...,x_{d-1})$.  However, this field has transcendence degree
  $d-1$ over $\R$, so there must be a non-trivial algebraic relation
  between them; i.e., there exists $g\in \R[z_1,...,z_d]$ with
  $g(\phi_1,...,\phi_d)=0$.
\end{proof}

\begin{lemma}\label{l:poly}
  Suppose that $m\in \N$; $n=2m-3$. Then there exists a
  polynomial~$g$ in $n$ variables $z_1,...,z_n$ with the following
  property. For any configuration of $m$ points in the plane realising
  angles $\theta_1,\theta_2,\dots,\theta_n$ we have
  \[
  g(s(\theta_1),s(\theta_2),\dots,s(\theta_n))=0
  \]
  where $s$ is the function defined by $s(\theta)=\sin^2\theta$.
\end{lemma}
\begin{proof}
  Suppose that $v_1,v_2,\dots,v_m$ are the $m$ points of
  the configuration and we may assume that $v_1$ and $v_2$ are fixed
  with $v_1=(0,0)$ and $v_2=(1,0)$.

  Let $\phi_1,\phi_2,\dots,\phi_k$ be all the $k=3\binom m3$ angles
  realised by these $m$ points. Each of these angles is a function of
  the positions of the three points forming that angle and, for each
  $1\le i\le k$, we let $f_i\colon \R^{2m-4} \to \R$ be the map that
  sends the $2m-4$ coordinates of the $m-2$ (non-fixed) points
  $v_3,v_4,\dots,v_m$ to $s(\phi_i)$. Since
  \[
  f_i(v_3,v_4,\dots,v_m)=s(\phi_i)=1-\cos^2\phi_i=1-\frac{((v_j-v_k).(v_l-v_k))^2}{\|v_j-v_k\|^2\|v_l-v_k\|^2},\]
  where $v_j,v_k,v_l$ are the triple giving angle $\phi_i$, we see
  that each $f_i$ is a rational function of the coordinates of the
  points.

  For each set $I\subset [k]$ with $|I|=2m-3$,
  Theorem~\ref{t:trans-deg} applied to the functions $f_i$, $i\in I$,
  implies that there exists a polynomial $g_I$ in
  $\R[x_1,x_2,\dots,x_{2m-3}]$ such that $g_I$ applied to $f_i$, $i\in
  I$, is identically zero.  In particular, if the angles
  $\theta_1,\theta_2,\dots,\theta_n$ chosen are exactly $\phi_i$,
  $i\in I$ then $g_I(s(\theta_1),s(\theta_2),\dots,s(\theta_n))=0$.
  Let $g=\prod_{I\subset[k];\ |I|=2m-3}g_I$ be the product of all the
  possible $g_I$. Obviously $g$ is zero for any configuration and any
  choice of $n=2m-3$ angles $\theta_1,\theta_2,\dots,\theta_n$ from
  among $\phi_1,\phi_2,\dots,\phi_k$.
\end{proof}

\begin{corollary}\label{c:2m-3}
  Suppose $m\in \N$ and $n=2m-3$.  Let $S\subset[0,\pi]^n$ be the set
  of $n$-tuples of angles that can be realised by $m$ points in the
  plane. Then the set $S$ has measure zero. Moreover $S$ has Hausdorff
  dimension $n-1$.
\end{corollary}
\begin{proof}
  That $S$ has measure follows immediately from Lemma~\ref{l:poly}, as
  does the fact that the Hausdorff dimension at most $n-1$. To
  complete the proof note that we can realise any $n$ angles with
  $\theta_1>\theta_2>\dots>\theta_n$ and $\theta_n=\theta_1-\theta_2$
  with $m$ points. Indeed, the construction (Lemma~\ref{l:5points-2d}
  together with Lemma~\ref{l:extend-2d}) shows realises the angles
  $\theta_1,\theta_2,\dots,\theta_{n-1}$ and we see that the angle
  $\theta_n=\theta_1-\theta_2$ also occurs in the construction.  Hence
  this co-dimension one subset is contained in $S$ and the result
  follows.
\end{proof}

\section{Higher dimensions}\label{s:hd-upper}
We start this section by examining the effect that a random projection
onto a two dimensional subspace has on an angle: in particular on very
small angles.  First we consider the length of a vector after a random
one-dimensional projection. Note, that both of the following two
lemmas are far from tight but suffice for our needs.

\begin{lemma}\label{l:project-vector}
  Suppose that $v$ is a unit vector in $\R^d$, and that $p$ is a
  projection onto a uniformly chosen one-dimensional subspace. Then,
  for all $\eps>0$
  \[
  \Prb(\|p(v)\|<\eps)\le \sqrt\eps+d\eps.
  \]
\end{lemma}
\begin{proof}
  Instead of the setup in the lemma we consider a fixed projection~$p$ on
  to the first coordinate and a random unit vector. Now one way of
  generating a random unit vector is to generate $d$ independent
  variables $X_i$ all with $N(0,1)$ distribution and then let $v=X/\|X\|$.

  If $\|p(v)\|<\eps$ then either $|X_1|<\sqrt\eps$ or $\|X\|\ge
  1/\sqrt\eps$. Since the density function for the normal distribution
  is always less than $1/2$ we see that
  $\Prb(|X_1|<\sqrt\eps)<\sqrt\eps$. Also $\|X\|^2$ is a random variable
  with mean $d$. Hence, by Markov's inequality,
  \[
  \Prb(\|X\|\ge 1/\sqrt\eps)=  \Prb(\|X\|^2\ge 1/\eps)\le d\eps
  \]
  Thus 
  \[
  \Prb(v_1<\eps)\le \sqrt\eps+d\eps.\qedhere
  \]
\end{proof}

Next we show that a random projection does not change angles `too
much'.  
[In the following lemma we refer to a `uniformly chosen
two-dimensional subspace'. Formally, this refers to the probability
measure on the space of all two-dimensional subspaces that is induced
by the action of $SO(d)$ equipped with its standard Haar measure.]

\begin{lemma}\label{l:project-angle}
  Suppose that $v_1,v_2$ are two vectors with angle $\theta<\pi/3$
  between them, and that $p$ is a projection onto a uniformly chosen
  two-dimensional subspace. Let $\phi$ be the angle between $p(v_1)$
  and $p(v_2)$. Then
  \[
  \Prb(\phi\not\in[\theta\eps,\theta/\eps])<3\sqrt\eps+5d\eps.
  \]
\end{lemma}
\begin{proof}
  Similarly to the previous proof, we will fix the projection $p$, this
  time onto the first two coordinates, and choose random unit vectors
  $v_1$ and $v_2$ subject to the constraint that the angle between
  them is $\theta$. We do this by picking a random unit vector $v$ and
  a random unit vector $u$ orthogonal to $v$, and setting $v_1=v$ and
  $v_2=\alpha v+\beta u$ where $\alpha=\cos \theta$ and $\beta=\sin
  \theta$. Since $\theta<\pi/3$ we have that $\alpha>1/2$, and
  $\theta<\tan\theta=\beta/\alpha<2\theta$.

  We need to bound the angle $\phi$ between $p(v_1)$ and
  $p(v_2)=\alpha p(v_1)+\beta p(u)$. First we bound the probability
  that this angle is big. The angle $\phi$ satisfies 
  \[
\phi\le  \sin^{-1}\left(\frac{\beta\|p(u)\|}{\alpha\|p(v_1)\|}\right)
  \le \frac{2\beta}{\alpha\|p(v_1)\|}
  \le \frac{4\theta}{\|p(v_1)\|}. 
  \]
  Thus, by Lemma~\ref{l:project-vector} for any $\eps>0$,
  \[
  \Prb(\phi>\theta/\eps)\le\Prb(\|p(v_1)\|<4\eps)\le2\sqrt \eps+4d\eps.
  \]

  Next we bound the chance that $\phi$ is small. The vector $u$ is
  chosen orthogonal to $v$ so lies on a $d-1$ dimensional sphere. Take
  a basis of this sphere such that the first coordinate direction is
  in the plane spanned by the first two coordinate directions and
  perpendicular to $p(v_1)$. Now the component of $p(u)$ perpendicular
  to $p(v)$ is exactly the first component $u_1$ of $u$ in this
  basis. Since $\|p(v)\|\le 1$ we have
  \[\phi\ge \tan^{-1}\left(\frac{\beta|u_1|}{\alpha\|p(v)\|}\right)\ge |u_1|\theta
  \]
  Applying Lemma~\ref{l:project-vector} again we see that, for any
  $\eps>0$,
  \[
  \Prb(\phi<\theta\eps)\le\Prb(|u_1|<\eps)\le \sqrt \eps+(d-1)\eps.
  \]
  Combining these two bounds gives the result.
\end{proof}
This result shows that projecting on to a random plane only changes
the angles by at most a constant factor. We want to use the upper
bound of $2m-4$ for the number of angles that can be realised by $m$
points in the plane to deduce the same bound in higher
dimensions. Since we do not have control over exactly what happens to
the angles when we project we need a bound for the two dimensional
case that can cope with the uncertainty introduced by the projection.
We use Lemma~\ref{l:poly} to prove such a result.
\begin{theorem}\label{t:2D-not-dense}
  Suppose that $h$ is a strictly positive real valued function, $m\in
  \N$, $n=2m-3$, and that $\eps>0$ is given.  Then there exists
  $\eps_i\le \eps$ for $1\le i\le n$ such that for any sequence of
  angles $\theta_i$ for $1\le i\le n$ with
  $\theta_i\in[h(\eps_i),\eps_i]$ there is no arrangement of $m$
  points in the plane achieving the angles $\theta_i$,
  $i=1,\dots,n$.
\end{theorem}
\begin{proof}
  Let $g$ be the polynomial given by Lemma~\ref{l:poly}.  Consider the
  terms of~$g$ and order them reverse lexicographically so that the
  leading term has the lowest power of $z_n$, and among those, is
  the one with lowest power of $z_{n-1}$, etc.

  As above define $s\colon (0,\pi)\to (0,1]$ by $s(x)=\sin^2(x)$. Then
  $s$ is a bijection on $(0,\pi/2]$, and whenever we write $s^{-1}$ we
  mean the preimage in $(0,\pi/2]$. Let $\tilde h=s\circ h\circ s^{-1}$ and
  note that $\tilde h$ is strictly positive.

  Assume, without loss of generality, that the leading term is
  $z_1^{a_1}\dots z_n^{a_n}$ (with coefficient $1$). Let $C\ge 1$ be
larger than the sum of the absolute values of all the coefficients of
the other terms. Set $\delta=s(\eps)$ and
$\delta_1=\min(1/C,\delta)$. Inductively define
  \[
  \delta_i=\min\left(\delta,\ \frac{\prod_{j<i}\tilde h(\delta_j)^{a_j}}{C}\right).
  \]
  Then for any $z=(z_1,z_2,\dots, z_n)$ with $z_i\in
  [\tilde h(\delta_i),\delta_i]$ the leading term is larger than the sum of
  all the other terms.  Indeed, by our ordering of terms, for any
  other term $z_1^{b_1}\dots z_n^{b_n}$ there exists $k$ such that
  $b_k>a_k$ and $b_i=a_i$ for $i>k$. Thus, by the definition of
  $\delta_i$ and that $1/C<1$,
  \[
  z_1^{b_1}\dots z_n^{b_n}\le z_k^{b_k}\dots z_n^{b_n}\le
  \frac{z_k}{z_1^{a_1}\dots z_{k-1}^{a_{k-1}}}z_1^{a_1}\dots z_n^{a_n}
  \le \frac{1}{C}z_1^{a_1}\dots z_n^{a_n}.
  \]
  In particular this shows that $g(z)\not=0$.

  Define $\eps_i=s^{-1}(\delta_i)$. By definition $\delta_i\le \delta$
  so, since $s^{-1}$ is increasing, $\eps_i\le \eps$. Moreover, any
  configuration with angles $\theta_i\in [h(\eps_i),\eps_i]$ would
  have $s(\theta_i)\in [s\circ h(\eps_i),s(\eps_i)]=[\tilde h(\delta_i),\delta_i]$ and thus
  \[g(s(\theta_1),s(\theta_2),\dots,s(\theta_n))\not=0\] which would
  contradict Lemma~\ref{l:poly}. Therefore there cannot be a
  configuration realising all the angles $\theta_i$, $1\le i\le n$.
\end{proof}
We are now in a position to prove Theorem~\ref{t:main-hd}; we just
need to combine Lemma~\ref{l:project-angle} and
Theorem~\ref{t:2D-not-dense}.
\begin{proof}[Proof of Theorem~\ref{t:main-hd}]
  Recall that $m\ge 2$ and $n=2m-3$. Trivially, for $m=2$ there is
  nothing to prove so we may assume $m\ge 3$.  We start with a trivial
  observation: however we place $m$ points in Euclidean space they
  actually lie in an $(m-1)$-dimensional (affine) subspace. Thus, it
  is sufficient to prove that the angles cannot be realised in
  $\R^{m-1}$ or, as we shall prove, in $\R^m$.

  Let $\delta$ be such that $3\sqrt\delta+5m\delta <1/2n$ (for example
  $\delta=1/100n^2$ will do).  Fix $\eps<\pi/3$ and define
  $h(x)=\delta^2x$. Now apply Theorem~\ref{t:2D-not-dense} to get a
  sequence $\eps_i$ for $1\le i\le n$ such that for any sequence
  $\theta_i$ with $h(\eps_i)<\theta_i<\eps_i$ there is no arrangement
  of $m$ points in the plane realising these angles.

  Let $\phi_i=\eps_i\delta$. We claim that the set of angles
  $\phi_i$ is not realisable by $m$ points
  in $\R^m$.  Indeed, suppose that there is such an arrangement of $m$
  points realising these angles.  Consider a random two-dimensional
  projection $p$ of these points. With a slight abuse of notation for
  an angle $\phi$ we use $p(\phi)$ to denote the angle between the
  lines after the projection. By our choice of $h$ and
  Lemma~\ref{l:project-angle} we see that
  \[
  \Prb(p(\phi_i)\not\in [h(\eps_i),\eps_i]) = \Prb(p(\phi_i)\not\in
  [\phi_i \delta,\phi_i/\delta])<3\sqrt\delta+5m\delta <1/2n.
  \]
  Thus, with positive probability we have $p(\phi_i)\in
  [h(\eps_i),\eps_i]$ for all $i$. Fix such a projection $p$. The $m$
  points $p(z_i)$ in the plane realise the $2m-3$ angles $p(\phi_i)$
  which contradicts Theorem~\ref{t:2D-not-dense}.
\end{proof}

Finally, we remark that a slight adjustment to the proof above shows
not only that there is a set of $n$ angles that is not realisable, but
also that the set of $n$-tuples that are not realisable has non-empty
interior.

\section{Angles bounded away from $0$ and $\pi$}\label{s:away-from-zero}
In this section we prove Theorem~\ref{t:away-from-zero}. We remark
that the constraint on $\theta_i$ in Theorem~\ref{t:2D-not-dense} was
really a constraint on $\sin \theta_i$. Thus, our proof of
Theorem~\ref{t:main-hd} shows not only that we cannot realise sets
containing very small angles but, also, that we cannot realise sets
containing very large angles (i.e. close to $\pi$). Hence it is
natural to bound the angles away from $0$ and $\pi$ as we do in
Theorem~\ref{t:away-from-zero}.

First we show a local result: that for any $\theta\in(0,\pi)$ we can
realise angles near $\theta$ efficiently.
\begin{lemma}
  Given $d$ and $\theta$ there exists a finite collection of points
  $A$ in $\R^d$ and an open set $U\subset \R$ containing $\theta$ such
  that, for any $k$, we can realise any $dk$ distinct angles all in
  $U$ with $k$ points in $\R^d$ together with $A$.
\end{lemma}
  The idea is to choose our set $A$ such that it subtends $d$ angles
  of size $\theta$ at the origin. Then if we moved the
  point at the origin a small amount we could modify these $d$ angles
  slightly and use the inverse function theorem to show that these
  perturbations must include a small open set around $\theta$.
\begin{proof}
  We start by defining the set $A$.  Fix $\lambda$ very large and let
  $a=\lambda\cos \theta$ and $b=\frac{1}{\sqrt{d-1}}\lambda \sin
  \theta$. For $1\le j\le d$ define $f_j=\sum_i be_i+(a-b)e_j$ where
  $e_1,e_2,\dots, e_d$ denote the standard basis vectors of $\R^d$.
  Let $A$ be the set
  \[
  \{e_1,e_2,\dots,e_d\}\cup \{f_1,f_2,\dots,f_d\}.
  \] 
  Since 
  \[\cos(e_i\widehat 0f_i)=\frac{e_i\cdot
    f_i}{\|e_i\|\|f_i\|}=\frac{a}{\sqrt{a^2+(d-1)b^2}}=\frac{\lambda\cos
    \theta}{\lambda}=\cos \theta
  \]
  we see that, for each $1\le i\le d$ we have  $e_i\widehat 0f_i=\theta$.

  Let $x=(x_1,x_2,\dots,x_d)$ be a point near $0$ and let $F$ be the
  function mapping $x$ to the $d$-tuple $(\cos (e_1\widehat xf_1),\cos
  (e_2\widehat xf_2),\dots,\cos (e_d\widehat xf_d))$.  Obviously the
  function $f$ is continuously differentiable near zero. Thus, to
  complete the proof we just need to show that the derivative of $f$
  is non-singular at $0$.

We have 
\begin{align*}
 \left.\frac{\partial }{\partial x_j}\cos\theta_i\right\vert_{x=0}
&=\left.\frac{\partial }{\partial x_j}\left(\frac{f_i-x}{\|f_i-x\|}\cdot\frac{e_i-x}{\|e_i-x\|}\right)\right\vert_{x=0}\\
&=\left.\frac{\partial }{\partial x_j} \frac{f_i-x}{\|f_i-x\|}\right\vert_{x=0}\cdot \frac{e_i-x}{\|e_i-x\|}+
\frac{f_i-x}{\|f_i-x\|}\cdot \left.\frac{\partial }{\partial x_j} \frac{e_i-x}{\|e_i-x\|}\right\vert_{x=0}\\
&=\left.\frac{\partial }{\partial x_j} \frac{f_i-x}{\|f_i-x\|}\right\vert_{x=0}\cdot e_i+
\frac{f_i}{\|f_i\|}\cdot \left.\frac{\partial }{\partial x_j} \frac{e_i-x}{\|e_i-x\|}\right\vert_{x=0}.
\end{align*}
Now \[\left\|\frac{\partial }{\partial x_j}
\frac{f_i-x}{\|f_i-x\|}\right\|=O(1/\lambda)\] and
\[\frac{\partial }{\partial x_j}
\frac{e_i-x}{\|e_i-x\|}=\begin{cases}0\qquad&\text{if $i=j$}\\
e_j\qquad&\text{otherwise}
\end{cases}\]
Thus
\[
\frac{\partial }{\partial x_j}\cos\theta_i=\begin{cases}O(1/\lambda)\qquad&\text{if $i=j$}\\
\frac{b}{\lambda}+O(1/\lambda)\qquad&\text{otherwise}.
\end{cases}
\]
Hence the derivative matrix of $F$ is 
\[\frac{b}{\lambda} (J-I)+O(1/\lambda)=\frac{\sin \theta}{\sqrt{d-1}}(J-I)+O(1/\lambda)
\]
(where $J$ is the all one matrix), and is thus invertible provided
$\lambda$ is sufficiently large. The result follows.
\end{proof}

Finally we prove Theorem~\ref{t:away-from-zero}.
\begin{proof}[Proof of Theorem~\ref{t:away-from-zero}]
  First we consider the case where all angles are distinct.

  For each $\phi\in [\eps,\pi-\eps]$ let $U_\phi$ and $A_\phi$ be the
  open neighbourhoods and sets of points given by the previous
  lemma. The $U_\phi$ form an open cover of $[\eps,\pi-\eps]$. Let
  $U_{\phi_1},U_{\phi_2},\dots,U_{\phi_k}$ be a finite subcover and
  let $A$ be the union of the corresponding $A_{\phi_i}$. Note, in
  particular, that $A$ is a finite set.

  We claim that we can realise any collection of $n$ angles
  $\theta_1,\theta_2,\dots,\theta_n$ with $n/d+k+|A|$ points. First we
  partition the angles $\theta_1,\theta_2,\dots,\theta_n$ into sets
  $\Phi_i$ such that all angles in $\Phi_i$ are in $U_{\phi_i}$.  Now,
  we place all the points in $A$ and then, for each $i$, use the
  previous lemma to realise all the angles in $\Phi_i$ using
  $\Big\lceil|\Phi_i|/d\Big\rceil$ more points. In total this uses
  (at most) $|A|+k+n/d$ points. This rearranges to the claimed bound.

  Now, suppose that there are some repeated angles. If any angle
  occurs more than $2d(2d+1)$ times then, by Lemma~\ref{l:eq-angles},
  we can realise it using $2(2d+1)$ points. Repeating this argument
  reduces to the case where no angle occurs more than $2d(2d+1)$
  times. This final case is easy: we take $2d(2d+1)$ copies of the
  above construction. Thus, we get a bound of $2d(2d+1)(|A|+k)+n/d$ as
  required.
\end{proof}

\section{Open Problems}

It is clear that some sets of more than $2m-4$ angles are realisable
by $m$ points in $\R^d$. Our main question asks how common this is.

To formalise this suppose that $d\ge 2$, $m\ge 2$ and $n\le
dm-\binom{d+1}{2}-1$, and that $\theta_1,\theta_2,\dots,\theta_n$ are
$n$ angles chosen uniformly from $(0,\pi)$. Let $P(d,m,n)$ be the
probability that the angles $\theta_1,\theta_2,\dots,\theta_n$ are
realisable with $m$ points in $\R^d$. Then our question becomes `how
does $P$ behave?'

We start with the case where $d$ is fixed.
\begin{question}
  Suppose $d$ is fixed and $n=dm-\binom{d+1}{2}-1$. Does
  $P(d,m,n)\to1$ as $m\to \infty$? More weakly, is $\liminf P(d,m,n)>0$? 
\end{question}
We believe that this probability does tend to 1. The situation is less
clear when $d$ is allowed to vary. We ask about the case when $m=d+1$
(i.e., the minimal number of points to actually use $d$ dimensions).
\begin{question}
  Suppose $m=d+1$ and $n=dm-\binom{d+1}{2}-1$. Does $P(d,m,n)\to 1$ as
  $d\to \infty$?
\end{question}

In the previous section we considered the case of angles bounded away
from $0$ and $\pi$. Our lower bound there was any $n$ angles can be
realised with $m$ points provided that $n\le dm+c$. The degrees of
freedom upper bound (i.e. the higher dimensional analogue of
Corollary~\ref{c:2m-3}) shows that we cannot hope to realise more than
$n=dm-\binom{d+1}{2}-1$ in general. These bounds have the same order
but differ by a constant.

\begin{question}
  Suppose that $\eps$ and $d$ are fixed. Is it possible to realise an
  arbitrary set of $n=dm-\binom{d+1}{2}-1$ angles all between $\eps$
  and $\pi-\eps$ with $m$ points for all sufficiently large $m$? 
\end{question}
Of course, if the answer to this is negative then one would like to
determine the correct value of the constant.

Our final question asks about the number of different ways that a
set of angles can be realised. We consider two similar point sets
(i.e, two point sets related by a translation, rotation, reflection or
scaling) as the same.
\begin{question}
  Suppose that $\theta_1,\theta_2,\dots,\theta_n$ are $n$ angles and
  suppose that $m$ is such that $n=2m-4$. How many different $m$-point
  sets are there realising these angles?
\end{question}
By Theorem~\ref{t:2d-opt} we know that there is at least one $m$-point
configuration realising these angles. By the remark following
Lemma~\ref{l:extend-2d} we see that there are at least $\Omega(k)$
chords we could use to realise the $k^\textrm{th}$ angle. This gives a
lower bound of roughly the order of $m!$ on the number of configurations. 

In the other direction, Lemma~\ref{l:poly} shows that, for almost all
tuples $(\theta_1,\theta_2,\dots,\theta_n)$, there are only finitely
many other angles that can occur in configurations realising these
angles. This shows that the number of such configurations is almost
surely finite.  By considering the possible arrangements of the points
one can check that $m^{10m}$ is an upper bound. But we do not know the
correct order.

 \bibliography{mybib}{}
\bibliographystyle{abbrv}

\end{document}